\documentclass[12pt]{article}
\usepackage{amssymb, amsmath, latexsym, a4}
\usepackage[latin1]{inputenc}
\usepackage{graphicx}
\usepackage[all]{xy}
\usepackage[usenames]{color}

\newenvironment{proof}[1][Proof]{\noindent\textbf{#1.} }{\ \rule{0.5em}{0.5em}}

\newtheorem{De}{Definition}[section]
\newtheorem{Th}[De]{Theorem}
\newtheorem{Pro}[De]{Proposition}
\newtheorem{Le}[De]{Lemma}
\newtheorem{Co}[De]{Corollary}
\newtheorem{Rem}[De]{Remark}
\newtheorem{Ex}[De]{Example}

\newcommand{\Image}{{\sf Im}}
\newcommand{\Ker}{{\sf Ker}}

\newcommand{\Lie}{\ensuremath{\mathsf{Lie}}}

\newcommand{\Lieh}{\ensuremath{\mathfrak{h}}}
\newcommand{\Lieg}{\ensuremath{\mathfrak{g}}}
\newcommand{\Liet}{\ensuremath{\mathfrak{t}}}
\newcommand{\Lieq}{\ensuremath{\mathfrak{q}}}

\newcommand{\Liem}{\ensuremath{\mathfrak{m}}}
\newcommand{\LieK}{\ensuremath{\mathcal{K}}}
\newcommand{\LieL}{\ensuremath{\mathcal{L}}}
\newcommand{\LieH}{\ensuremath{\mathcal{H}}}
\newcommand{\LieS}{\ensuremath{\mathfrak{S}}}
\newcommand{\LieM}{\ensuremath{\mathfrak{M}}}
\newcommand{\LieI}{\ensuremath{\mathfrak{I}}}
\newcommand{\Lien}{\ensuremath{\mathfrak{n}}}

\newcommand{\Lies}{\ensuremath{\mathfrak{s}}}

\newcommand{\Leib}{\ensuremath{\mathsf{Leib}}}

\newcommand{\ze}{{\cal Z}}

\newbox\pullbackbox
\setbox\pullbackbox=\hbox{\xy 0;<1mm,0mm>: \POS(4,0)\ar@{-} (0,0) \ar@{-} (4,4)
\endxy}

\newdir{>>}{{}*!/3.5pt/:(1,-.2)@^{>}*!/3.5pt/:(1,+.2)@_{>}*!/7pt/:(1,-.2)@^{>}*!/7pt/:(1,+.2)@_{>}}
\newdir{ >>}{{}*!/8pt/@{|}*!/3.5pt/:(1,-.2)@^{>}*!/3.5pt/:(1,+.2)@_{>}}
\newdir{ |>}{{}*!/-3.5pt/@{|}*!/-8pt/:(1,-.2)@^{>}*!/-8pt/:(1,+.2)@_{>}}
\newdir{ >}{{}*!/-8pt/@{>}}
\newdir{>}{{}*:(1,-.2)@^{>}*:(1,+.2)@_{>}}
\newdir{<}{{}*:(1,+.2)@^{<}*:(1,-.2)@_{<}}

  \newcommand{\eh}{\frak h}

\begin{document}

\centerline{\bf  A STUDY OF  $n$-$\Lie$-ISOCLINIC LEIBNIZ ALGEBRAS}

\bigskip
\centerline{\bf G. R. Biyogmam$^1$ and J. M. Casas$^2$}

\bigskip
\centerline{$^1$ Department of Mathematics, Georgia College \& State University}
\centerline{Campus Box 17 Milledgeville, GA 31061-0490}
\centerline{ {E-mail address}: guy.biyogmam@gcsu.edu}
\bigskip

\centerline{$^2$Dpto. Matemática Aplicada, Universidad de Vigo,  E. E. Forestal}
\centerline{Campus Universitario A Xunqueira, 36005 Pontevedra, Spain}
\centerline{ {E-mail address}: jmcasas@uvigo.es}
\bigskip

\date{}

\bigskip \bigskip \bigskip

{\bf Abstract:}  In this paper we  introduce the concept of $n$-\Lie-isoclinism on non-Lie Leibniz algebras. Among the results obtained, we provide several characterizations of  $n$-\Lie-isoclinic classes of Leibniz algebras.
Also, we provide a characterization of $n$-\Lie-stem Leibniz algebras, and prove that every $n$-{\Lie}-isoclinic class of Leibniz algebras contains a $n$-\Lie-stem Leibniz algebra.

\bigskip

\bigskip



\section{Introduction}

The concept of isoclinism was introduced  in 1939 by P. Hall in his  attempt to classify $p$-groups using an equivalence relation weaker than the notion of isomorphism \cite{PH}.
Later himself generalized the notion of isoclinism to that of isologism in \cite{PH1}, which is in fact isoclinism with respect to a certain variety of groups. When the
variety of all the trivial groups is considered, then the  notion of isomorphism is recovered. When  the chosen variety is the variety of all abelian groups, then the notion of isoclinism is recovered. When  the variety of all nilpotent groups of class at most $n$ is considered, then arises the  notion  of $n$-isoclinism.

 This concept has been widely studied in several algebraic structures. For example, it is worth mentioning \cite{Bio, ERR, He, MSCh, vdW} in case of groups; \cite{HG} in case of pairs of groups; in case of Lie algebras \cite{PMKh, Sal3} and \cite{MSTV} in the case of pairs of Lie algebras.

  Recently,  the concept of isoclinism of Lie algebras has been considered in the relative context, given rise to the notion of {\Lie}-isoclinism  of Leibniz algebras \cite{BC}.  Relative means that the  notion of {\Lie}-isoclinism arises through the Liezation functor  $(-)_{\Lie}: {\sf Leib} \to {\sf Lie}$ which assigns the Lie algebra $\Lieg_{\Lie} = \Lieg/<\{[x,x]: x \in \Lieg\}>$ to a given Leibniz algebra $\Lieg$, while the classical notion of isoclinism arises through the abelianization functor, which assigns to a Leibniz algebra {\Lieg} the abelian Leibniz algebra $\frac{\Lieg}{[{\Lieg}, {\Lieg}]}$. This philosophy comes from the categorical theory of central extensions relative to a chosen Birkhoff subcategory of a semi-abelian category. We refer to  \cite{BC, CKh, CVDL1} and references given therein for a detailed explanation.

Our goal in this article is to study the notion of $n$-isoclinism relative to the Liezation functor, that is we consider the variety  of {\Lie}-nilpotent non-Lie Leibniz algebras of class at most $n$, called as $n$-{\Lie}-isoclinism, which is an equivalence relation between two Leibniz algebras $\Lieg_1$ and $\Lieg_2$ for which there exist two isomorphisms $\eta : \frac{\Lieg_1}{{\ze}_{n}^{\Lie}({\Lieg_1})} \to \frac{\Lieg_2}{{\ze}_{n}^{\Lie}({\Lieg_2})}$ and $\xi : \gamma_{n+1}^{\Lie}(\Lieg_1) \to \gamma_{n+1}^{\Lie}(\Lieg_2)$ such that the following diagram is commutative:
\begin{equation*}  \label{square isoclinic}
\xymatrix{
\frac{\Lieg_1}{{\ze}_{n}^{\Lie}({\Lieg_1})}  \times \stackrel{n+1} \ldots\times \frac{\Lieg_1}{{\ze}_{n}^{\Lie}({\Lieg_1})}   \ar[r]^-*+{C_1^{n+1}} \ar[d]_{\eta^{ n+1}}  & \gamma_{n+1}^{\Lie}(\Lieg_1)  \ar[d]^{\xi}\\
\frac{\Lieg_2}{{\ze}_{n}^{\Lie}({\Lieg_2})}  \times \stackrel{n+1} \ldots\times \frac{\Lieg_2}{{\ze}_{n}^{\Lie}({\Lieg_2})}    \ar[r]^-*+{C_2^{n+1}} & \gamma_{n+1}^{\Lie}(\Lieg_2) }
\end{equation*}
where  $C_1^{n+1}(\bar{x}_1,\ldots,\bar{x}_{n+1})=[[[x_1,x_2]_{lie},x_3]_{lie},\ldots,x_{n+1}]_{lie}$ with $x_1,\ldots,x_{n+1}\in\Lieg_1$ and   $\bar{x}:=x+ {\ze}_{n}^{\Lie}({\Lieg_1}),$ and $C_2^{n+1}(\tilde{y}_1,\ldots,\tilde{y}_{n+1})=[[[y_1,y_2]_{lie},y_3]_{lie},\ldots,y_{n+1}]_{lie}$ with \\$y_1,\ldots,y_{n+1}\in\Lieg_2$  and $\tilde{y}:=y+ {\ze}_{n}^{\Lie}({\Lieg_2})$ (see Definitions \ref{isoclinic} and \ref{iso ext}).

This paper is organized as follows: In section \ref{preliminaries}, we present some generalities and preliminaries. In section \ref{Lie iso Lb alg}, we define the notion of $n$-{\Lie} isoclinism between {\Lie}-central extensions of Leibniz algebras and we prove  the corresponding relative results and characterizations of  several classical results on isoclinism of Lie algebras  in this framework  such as any Leibniz algebra is $n$-{\Lie}-isoclinic to some Leibniz algebra {\Lieh} satisfying $Z_{\Lie}({\Lieh}) \cap \gamma_n^{\Lie}({\Lieh}) \subseteq \gamma_{n+1}^{\Lie}({\Lieh})$  or a homomorphism of $\Lie$-central extensions $(\alpha, \beta, \gamma) : (g_1) \to (g_2)$ is  $n$-$\Lie$-isoclinic if and only if $\gamma$ is an isomorphism and $\Ker(\beta) \cap \gamma_{n+1}^{\Lie}(\Lieg_1) =0$. In section \ref{stem section}, we study properties of $n$-{\Lie} isoclinism of {\Lie}-stem Leibniz algebras proving that every Leibniz algebra is $n$-{\Lie}-iscolinic to some $n$-{\Lie}-stem Leibniz algebra or that two $n$-{\Lie}-isoclinic $n$-{\Lie}-stem  Leibniz algebras have isomorphic $n$-{\Lie} centers.


\section{Preliminary results on Leibniz algebras} \label{preliminaries}
Let $\mathbb{K}$ be a fix ground field such that $\frac{1}{2} \in \mathbb{K}$. Throughout the paper, all vector spaces and tensor products are considered over $\mathbb{K}$.

A \emph{Leibniz algebra} \cite{Lo 1, Lo 2} is a vector space ${\Lieg}$  equipped with a bilinear map $[-,-] : \Lieg \otimes \Lieg \to \Lieg$, usually called the \emph{Leibniz bracket} of ${\Lieg}$,  satisfying the \emph{Leibniz identity}:
\[
 [x,[y,z]]= [[x,y],z]-[[x,z],y], \quad x, y, z \in \Lieg.
\]

 A subalgebra ${\eh}$ of a Leibniz algebra ${\Lieg}$ is said to be \emph{left (resp. right) ideal} of ${\Lieg}$ if $ [h,g]\in {\eh}$  (resp.  $ [g,h]\in {\eh}$), for all $h \in {\eh}$, $g \in {\Lieg}$. If ${\eh}$ is both
left and right ideal, then ${\eh}$ is called \emph{two-sided ideal} of ${\Lieg}$. In this case $\Lieg/\Lieh$ naturally inherits a Leibniz algebra structure.

Given a Leibniz algebra ${\Lieg}$, we denote by ${\Lieg}^{\rm ann}$ the subspace of ${\Lieg}$ spanned by all elements of the form $[x,x]$, $x \in \Lieg$. It is clear that the quotient ${\Lieg}_ {_{\rm Lie}}=\Lieg/{\Lieg}^{\rm ann}$ is a Lie algebra. This defines the so-called  \emph{Liezation functor} $(-)_{\Lie} : {\Leib} \to {\Lie}$, which assigns to a Leibniz algebra $\Lieg$ the Lie algebra ${\Lieg}_{_{\rm Lie}}$. Moreover, the canonical epimorphism  ${\Lieg} \twoheadrightarrow {\Lieg}_ {_{\rm Lie}}$ is universal among all homomorphisms from $\Lieg$ to a Lie algebra, implying that the Liezation functor is left adjoint to the inclusion functor $ {\Lie} \hookrightarrow {\Leib}$.

Given a Leibniz algebra $\Lieg,$ we define the bracket $$[-,-]_{lie}:\Lieg\to \Lieg, ~~ \text{by} ~~ [x,y]_{lie}=[x,y]+[y,x], ~~ \text{for} ~~ x,y\in\Lieg.$$

Let  ${\Liem}$, ${\Lien}$ be two-sided ideals of a Leibniz algebra  ${\Lieg}$. The following notions come from \cite{CKh}, which were derived from \cite{CVDL1}.

 The \emph{$\Lie$-commutator} of  ${\Liem}$ and  ${\Lien}$ is the two-sided ideal  of $\Lieg$
\[
[\Liem,\Lien]_{\Lie}= \langle \{[m,n]_{lie}, m \in \Liem, n \in \Lien \}\rangle.
\]

The \emph{$\Lie$-center} of the Leibniz algebra $\Lieg$ is the two-sided ideal
\[
Z_{\Lie}(\Lieg) =  \{ z\in \Lieg\,|\,\text{$[q,z]_{lie}=0$ for all $q\in \Lieg$}\}.
\]

 The \Lie-{\it centralizer} of ${\Liem}$ and ${\Lien}$ over  ${\Lieg}$ is
\[
C_{\Lieg}^{\Lie}({\Liem} , {\Lien}) = \{g \in {\Lieg} \mid  \; [g, m]_{lie} \in {\Lien}, \; \text{for all} \;
m \in {\Liem} \} \; .
\]

\begin{De} \cite{CKh}
Let ${\Lien}$ be a two-sided ideal of a Leibniz algebra $\Lieg$. The lower $\Lie$-central series of $\Lieg$ relative to ${\Lien}$ is the sequence
\[
\cdots \trianglelefteq {\gamma_i^{\Lie}(\Lieg,\Lien)} \trianglelefteq \cdots \trianglelefteq \gamma_2^{\Lie}(\Lieg,\Lien)  \trianglelefteq {\gamma_1^{\Lie}(\Lieg,\Lien)}
\]
of two-sided ideals of $\Lieg$ defined inductively by
\[
{\gamma_1^{\Lie}(\Lieg,\Lien)} = {\Lien} \quad \text{and} \quad \gamma_i^{\Lie}(\Lieg,\Lien) =[{\gamma_{i-1}^{\Lie}(\Lieg,\Lien)},{\Lieg}]_{\Lie}, \quad   i \geq 2.
\]
\end{De}

We use the notation $\gamma_i^{\Lie}(\Lieg)$ instead of $\gamma_i^{\Lie}(\Lieg,\Lieg), 1 \leq i \leq n$.

If $\varphi : \Lieg \to \Lieq$ is a homomorphism of Leibniz  such that $\varphi({\Liem}) \subseteq {\Lien}$, where ${\Liem}$ is a two-sided ideal of ${\Lieg}$ and ${\Lien}$ a two-sided ideal of ${\Lieq}$, then $\varphi(\gamma_{i}^{\Lie}({\Lieg}, {\Liem})) \subseteq \gamma_{i}^{\Lie}({\Lieq}, {\Lien}), i \geq 1$.

\begin{De}
The Leibniz algebra $\Lieg$ is said to be $\Lie$-nilpotent relative to $\Lien$ of class $c$ if\ $\gamma_{c+1}^{\Lie}(\Lieg, \Lien) = 0$ and $\gamma_c^{\Lie}(\Lieg, \Lien) \neq 0$.
\end{De}

\begin{De} \cite{CKh}
The upper $\Lie$-central series of a Leibniz algebra ${\Lieg}$ is the sequence of two-sided ideals, called $i$-{\Lie} centers, i=0, 1, 2,  \dots,
 \[
{\ze}_0^{\Lie}({\Lieg}) \trianglelefteq {\ze}_1^{\Lie}({\Lieg}) \trianglelefteq \cdots \trianglelefteq {\ze}_i^{\Lie}({\Lieg}) \trianglelefteq \cdots
\]
 defined inductively by
\[
{\ze}_0^{\Lie}({\Lieg}) = 0 \quad \text{and} \quad
 {\ze}_{i}^{\Lie}({\Lieg}) = C_{\Lieg}^{\Lie}({\Lieg},{\ze}_{i-1}^{\Lie}({\Lieg})) , \  i \geq 1 .
 \]
\end{De}

\begin{De} \cite[Definition 3.2, Proposition 3.4]{BC1}
An exact sequence of Leibniz algebras $0 \to \Lien \to \Lieg \stackrel{\pi} \to \Lieq \to 0$ is said to be \emph{$n$-$\Lie$-central} extension if $\gamma_{n+1}^{\Lie}({\Lieg}, {\Lien})=0$, equivalently
$\Lien \subseteq {\ze}_{n}^{\Lie}(\Lieg)$.
\end{De}


\section{$n$-$\Lie$-isoclinic Leibniz algebras} \label{Lie iso Lb alg}

Let ${\Lieg_1}$  and  ${\Lieg_2}$ be two Leibniz algebras, then we can construct the following $n$-$\Lie$-central extensions

\begin{equation} \label{Lie central extension}
(g_i) : 0 \to  {\ze}_{n}^{\Lie}({\Lieg_i})\stackrel{}\to {\Lieg}_i \stackrel{} \to {\Lieg_i}/ {\ze}_{n}^{\Lie}({\Lieg_i}) \to 0, ~ i=1, 2.
\end{equation}

\begin{De} \label{isoclinic}
The $\Lie$-central extensions $(g_1)$ and $(g_2)$ are said to be $n$-\Lie-isoclinic whenever there exist two isomorphisms $\eta : \frac{\Lieg_1}{{\ze}_{n}^{\Lie}({\Lieg_1})} \to \frac{\Lieg_2}{{\ze}_{n}^{\Lie}({\Lieg_2})}$ and $\xi : \gamma_{n+1}^{\Lie}(\Lieg_1) \to \gamma_{n+1}^{\Lie}(\Lieg_2)$ such that the following diagram is commutative:
\begin{equation}  \label{square isoclinic}
\xymatrix{
\frac{\Lieg_1}{{\ze}_{n}^{\Lie}({\Lieg_1})}  \times \stackrel{n+1} \ldots\times \frac{\Lieg_1}{{\ze}_{n}^{\Lie}({\Lieg_1})}   \ar[r]^-*+{C_1^{n+1}} \ar[d]_{\eta^{ n+1}}  & \gamma_{n+1}^{\Lie}(\Lieg_1)  \ar[d]^{\xi}\\
\frac{\Lieg_2}{{\ze}_{n}^{\Lie}({\Lieg_2})}  \times \stackrel{n+1} \ldots\times \frac{\Lieg_2}{{\ze}_{n}^{\Lie}({\Lieg_2})}    \ar[r]^-*+{C_2^{n+1}} & \gamma_{n+1}^{\Lie}(\Lieg_2) }
\end{equation}
where  $C_1^{n+1}(\bar{x}_1,\ldots,\bar{x}_{n+1})=[[[x_1,x_2]_{lie},x_3]_{lie},\ldots,x_{n+1}]_{lie}$ with $x_1,\ldots,x_{n+1}\in\Lieg_1$ and   $\bar{x}:=x+ {\ze}_{n}^{\Lie}({\Lieg_1}),$ and $C_2^{n+1}(\tilde{y}_1,\ldots,\tilde{y}_{n+1})=[[[y_1,y_2]_{lie},y_3]_{lie},\ldots,y_{n+1}]_{lie}$ with \\$y_1,\ldots,y_{n+1}\in\Lieg_2$  and $\tilde{y}:=y+ {\ze}_{n}^{\Lie}({\Lieg_2}).$ So the map $\xi$ is given by $\xi([[[x_1,x_2]_{lie},x_3]_{lie},$ $\ldots,x_{n+1}]_{lie})=[[[y_1,y_2]_{lie},y_3]_{lie},\ldots,y_{n+1}]_{lie}$, with $x_i\in\Lieg_1$ and $y_i\in\eta(\bar{x}_i)$, for $i=1,\ldots,n+1.$

The pair $(\eta, \xi)$ is called a $n$-\Lie-isoclinism from $(g_1)$ to $(g_2)$, and will be denoted by $(\eta, \xi) : (g_1) \to (g_2)$.
\end{De}

\begin{De} \label{iso ext}
Let $\Lieg_1$ and $\Lieg_2$ be Leibniz algebras. We say that $\Lieg_1$ and $\Lieg_2$ are  $n$-$\Lie$-isoclinic if $(g_1)$ and $(g_2)$ are $n$-$\Lie$-isoclinic $n$-$\Lie$-central extensions. We denote it  by  $(\eta, \xi): {\Lieg}_1 \underset{n}\sim {\Lieg}_2$.
\end{De}

\begin{Rem}
Clearly,  $0$-\Lie-isoclinism coincides with isomorphism and $1$-\Lie-isoclinism coincides with  the notion of {\Lie}-isoclinism  given in \cite{BC}.
\end{Rem}

\begin{Pro} \label{equivalence relation}
$n$-$\Lie$-isoclinism is an equivalence relation.
\end{Pro}
\begin{proof}
The proof is straightforward
\end{proof}

\begin{Le}\label{Lem1}
Let $\Lieg$ be a Leibniz algebra, $\LieI$ an ideal of $\Lieg,$ and $\Lieh$ a subalgebra of $\Lieg.$ Then
\begin{enumerate}
\item[a)] $\Lieg$ is $n$-\Lie-isoclinic to $\Lieg\oplus\Lien,$ for all {\Lie}-nilpotent Leibniz algebra $\Lien$ of class  $n.$
\item[b)]  $\Lieh$  is $n$-\Lie-isoclinic to $\Lieh+{{\ze}_{n}^{\Lie}({\Lieg})}$ for $n \geq 1$ provided that  $[{\ze}_{n-1}^{\Lie}({\Lieh}), {\Lieh}]_{\Lie} \subseteq \gamma_{n+1}^{\Lie}(\Lieh)$.  Consequently, under the above requirements for $n$ and {\Lieh}, if $\Lieg=\Lieh+{{\ze}_{n}^{\Lie}({\Lieg})},$ then $\Lieg$ is $n$-\Lie-isoclinic to $\Lieh.$
\item[c)] If $\Lieg$ is $n$-\Lie-isoclinic to $\Lieh$  such that  $[{\ze}_{n-1}^{\Lie}({\Lieh}), {\Lieh}]_{\Lie} \subseteq \gamma_{n+1}^{\Lie}(\Lieh)$ and $\Lieg/{{\ze}_{n}^{\Lie}({\Lieg})} $ is finite dimensional, then $\Lieg=\Lieh+{{\ze}_{n}^{\Lie}({\Lieg})}.$
\item[d)] $\frac{\Lieg}{\LieI}$ is $n$-\Lie-isoclinic to $\frac{\Lieg}{\LieI\cap \gamma_{n+1}^{\Lie}(\Lieg)}.$ Consequently, if $\LieI\cap \gamma_{n+1}^{\Lie}(\Lieg)=0,$ then $\Lieg$ is $n$-\Lie-isoclinic to $ \frac{\Lieg}{\LieI}.$
\item[e)]  If  $\gamma_{n+1}^{\Lie}(\Lieg)$ is a finite dimensional subalgebra of $\Lieg$ and $\Lieg$ is $n$-\Lie-isoclinic to $ \frac{\Lieg}{\LieI},$ then $\LieI\cap \gamma_{n+1}^{\Lie}(\Lieg)=0.$
\end{enumerate}
\end{Le}
\begin{proof}
    To prove {\it a)}, let $\Lien$ be a {\Lie}-nilpotent Leibniz algebra of class $n.$ Then $\gamma_{n+1}^{\Lie}(\Lien)=0$ and ${{\ze}_{n}^{\Lie}(\Lien)}=\Lien.$ So ${{\ze}_{n}^{\Lie}(\Lieg\oplus\Lien)}={{\ze}_{n}^{\Lie}(\Lieg)}\oplus{{\ze}_{n}^{\Lie}(\Lien)}= {{\ze}_{n}^{\Lie}(\Lieg)} \oplus {\Lien},$ and $\gamma_{n+1}^{\Lie}(\Lieg\oplus\Lien)=\gamma_{n+1}^{\Lie}(\Lieg)\oplus\gamma_{n+1}^{\Lie}(\Lien)=\gamma_{n+1}^{\Lie}(\Lieg).$ Consider the maps $\eta:\frac{\Lieg}{{{\ze}_{n}^{\Lie}({\Lieg})}} \to \frac{\Lieg\oplus\Lien}{{{\ze}_{n}^{\Lie}({\Lieg})}\oplus\Lien}$ and $\xi=id_{\gamma_{n+1}^{\Lie}(\Lieg)},$ where $\eta(g+{{\ze}_{n}^{\Lie}(\Lieg)})=g+({{\ze}_{n}^{\Lie}(\Lieg)\oplus\Lien}).$ It is easy to check that $\eta$ and $\xi$ are isomorphisms and the corresponding diagram (\ref{square isoclinic}) commutes.
\medskip

To prove {\it b)}, set $\Lieh_1=\Lieh+{{\ze}_{n}^{\Lie}({\Lieg})}$ and consider the mapping $\eta:\Lieh/{{\ze}_{n}^{\Lie}({\Lieh})} \to \Lieh_1/{{\ze}_{n}^{\Lie}({\Lieh_1})}$ defined by $\eta(h+{{\ze}_{n}^{\Lie}(\Lieh)})=h+{{\ze}_{n}^{\Lie}(\Lieh_1)}.$
$\eta$ is clearly a surjective homomorphism. To show that $\eta$ is one-to-one, it is enough to show that  $\Lieh\cap{{\ze}_{n}^{\Lie}({\Lieh_1})}\subseteq{{\ze}_{n}^{\Lie}({\Lieh})}.$  We proceed by induction. Clearly, for $h\in \Lieh\cap{{\ze}_{1}^{\Lie}({\Lieh_1})},$ $[h,x]_{lie}=0$ for all $x\in\Lieh$ since $\Lieh\subseteq\Lieh_1,$ so $\Lieh\cap{{\ze}_{1}^{\Lie}({\Lieh_1})}\subseteq{{\ze}_{1}^{\Lie}({\Lieh})}.$ Now let $h\in\Lieh\cap{{\ze}_{n}^{\Lie}({\Lieh_1})},$ then for all $z\in\Lieh\subseteq\Lieh_1,$ $[h,z]_{lie}\in \Lieh\cap{{\ze}_{n-1}^{\Lie}({\Lieh_1})}\subseteq{{\ze}_{n-1}^{\Lie}({\Lieh})},$  implying that $h\in {{\ze}_{n}^{\Lie}({\Lieh})}.$ So $\Lieh\cap{{\ze}_{n}^{\Lie}({\Lieh_1})}\subseteq{{\ze}_{n}^{\Lie}({\Lieh})}.$ Therefore $\eta$ is an isomorphism.

 Then we claim  that $(\eta, id): (h_1)\underset{n}\sim (h)$ is a $n$-\Lie-isoclinism, with  the identity mapping $id: \gamma_{n}^{\Lie}(\Lieh)\to \gamma_{n}^{\Lie}(\Lieh_1).$
 Indeed, a standard induction shows that $\gamma_{k}^{\Lie}(\Lieh_1)\subseteq \gamma_{k}^{\Lie}(\Lieh)+{{\ze}_{n-1}^{\Lie}({\Lieg})}$  for all $2\leq k \leq n.$
We now have
$$
\begin{aligned}
\gamma_{n+1}^{\Lie}(\Lieh_1)&=[\gamma_{n}^{\Lie}(\Lieh_1),\Lieh_1]_{\Lie}\\&
=[\gamma_{n}^{\Lie}(\Lieh_1),\Lieh+{\ze}_{n}^{\Lie}({\Lieg})]_{\Lie} \\
&= [\gamma_{n}^{\Lie}(\Lieh_1),\Lieh]_{\Lie}~~\mbox{since}~~
[\gamma_{n}^{\Lie}(\Lieh_1), {{\ze}_{n}^{\Lie}({\Lieg})}]_{\Lie}=0 \\
&\subseteq\big[ \gamma_{n}^{\Lie}(\Lieh)+{{\ze}_{n-1}^{\Lie}({\Lieg})},~\Lieh\big]_{\Lie}\\
&\subseteq \big[ \gamma_{n}^{\Lie}(\Lieh),\Lieh\big]_{\Lie}+\big[{{\ze}_{n-1}^{\Lie}({\Lieg})},~\Lieh\big]_{\Lie}\\
&\subseteq \gamma_{n+1}^{\Lie}(\Lieh)+\big[{{\ze}_{n-1}^{\Lie}({\Lieh})},~\Lieh\big]_{\Lie}\\&\subseteq\gamma_{n+1}^{\Lie}(\Lieh).\end{aligned}
$$
The converse inclusion is obvious and the commutativity of diagram (\ref{square isoclinic}) is trivial.
\medskip

To prove {\it c)}, assume that $\Lieg\underset{n}\sim\Lieh$ and dim$\left( \frac{\Lieg}{{{\ze}_{n}^{\Lie}({\Lieg})}} \right)$ is finite. Again, set $\Lieh_1=\Lieh+{{\ze}_{n}^{\Lie}({\Lieg})}.$ Since by {\it b)}, $\Lieh\underset{n}\sim\Lieh_1,$ we have by Proposition \ref{equivalence relation} that $\Lieg\underset{n}\sim\Lieh_1$ which implies that $\frac{\Lieg}{{{\ze}_{n}^{\Lie}({\Lieg})}}\overset{\eta}\cong\frac{\Lieh_1}{{{\ze}_{n}^{\Lie}({\Lieh_1})}}.$ As ${{\ze}_{n}^{\Lie}({\Lieg})} \subseteq{{\ze}_{n}^{\Lie}({\Lieh_1})}, $ it follows that $${\rm dim} \left(\frac{\Lieg}{{{\ze}_{n}^{\Lie}({\Lieg})}} \right)={\rm dim} \left( \frac{\Lieh_1}{{{\ze}_{n}^{\Lie}({\Lieh_1})}} \right) \leq {\rm dim} \left( \frac{\Lieg}{{{\ze}_{n}^{\Lie}({\Lieh_1})}} \right) \leq {\rm dim} \left( \frac{\Lieg}{{{\ze}_{n}^{\Lie}({\Lieg})}} \right).$$ Since ${\rm dim}\left( \frac{\Lieg}{{{\ze}_{n}^{\Lie}({\Lieg})}} \right)$ is finite, it follows that $\Lieg \cong \Lieh_1 = \Lieh+{{\ze}_{n}^{\Lie}({\Lieg})}.$
\medskip

To prove {\it d)}, consider the map  $\eta: \Lieg_1/{{\ze}_{n}^{\Lie}({\Lieg_1})} \to\Lieg_2/{{\ze}_{n}^{\Lie}({\Lieg_2})}$ where $\Lieg_1=\frac{\Lieg}{\LieI}$ and $\Lieg_2=\frac{\Lieg}{\LieI\cap \gamma_{n+1}^{\Lie}(\Lieg)},$ defined by
$\eta(\bar{g}+{{\ze}_{n}^{\Lie}({\Lieg_1})} )=\tilde{g}+{{\ze}_{n}^{\Lie}({\Lieg_2})} $ with  $\bar{g}=g+\LieI$ and $\tilde{g}=g+\LieI\cap \gamma_{n+1}^{\Lie}(\Lieg).$
Then $\eta$ is clearly an isomorphism. Also the map $\xi: \gamma_{n+1}^{\Lie}(\Lieg_1)\to\gamma_{n+1}^{\Lie}(\Lieg_2)$ defined by $\xi\left(\left[ \left[ \left[\bar{g}_1, \bar{g}_2\right]_{lie}, \bar{g}_3 \right]_{lie}, \ldots,\bar{g}_{n+1} \right]_{lie} \right)=\left[ \left[ \left[\tilde{g}_1, \tilde{g}_2\right]_{lie}, \tilde{g}_3 \right]_{lie}, \ldots,\tilde{g}_{n+1} \right]_{lie},$ $g_1,\ldots,g_{n+1}\in\Lieg,$ is a well-defined isomorphism and the corresponding diagram (\ref{square isoclinic}) commutes. Therefore  $(\eta,\xi)$ is a $n$-\Lie-isoclinism.
\medskip

To prove {\it e)}, assume that dim$\left( \gamma_{n+1}^{\Lie}(\Lieg) \right)$ is finite and $\Lieg\underset{n}\sim\frac{\Lieg}{\LieI}.$ Then $\gamma_{n+1}^{\Lie}(\Lieg)\overset{\xi}\cong\gamma_{n+1}^{\Lie}(\frac{\Lieg}{\LieI}).$ Also, by {\it d)} we have $\frac{\Lieg}{\LieI}\underset{n}\sim\frac{\Lieg}{\LieI\cap \gamma_{n+1}^{\Lie}(\Lieg)}$ which yields $\gamma_{n+1}^{\Lie}(\frac{\Lieg}{\LieI})\overset{\xi'}\cong\gamma_{n+1}^{\Lie}(\frac{\Lieg}{\LieI\cap \gamma_{n+1}^{\Lie}(\Lieg)})\cong\frac{\gamma_{n+1}^{\Lie}(\Lieg)}{\LieI\cap \gamma_{n+1}^{\Lie}(\Lieg)}.$ Therefore $\gamma_{n+1}^{\Lie}(\Lieg)\cong\frac{\gamma_{n+1}^{\Lie}(\Lieg)}{\LieI\cap \gamma_{n+1}^{\Lie}(\Lieg)}.$ Hence $\LieI\cap \gamma_{n+1}^{\Lie}(\Lieg)=0$ as dim$ \left( \gamma_{n+1}^{\Lie}(\Lieg) \right)$ is finite.
\end{proof}

\begin{Ex} An example of the requirements in Lemma \ref{Lem1} {\it b)} is given by the  four-dimensional non-{\Lie}-nilpotent non-Lie Leibniz algebra
$\Lieh=span\{a_1,a_2,a_3,a_4\}$  with nonzero brackets $[a_1,a_1]=a_4, ~~[a_1,a_2]=-a_2,~~[a_2,a_1]=a_2,~~[a_3,a_1]=-a_3,~~[a_3,a_2]=a_4$ (algebra ${\cal L}_5$ in \cite[Proposition 3.9]{CK}).

Then clearly,  ${{\cal Z}}^{\Lie}_n(\Lieh)=span\{a_2,a_4\},~$ for all $n\geq 1.$ Also,   ${\gamma_n^{\Lie}(\Lieh)}=span\{a_3, a_4\}~$ for all $n\geq 2.$ One easily verifies that $\big[{{\ze}_{0}^{\Lie}({\Lieh})},~\Lieh\big]_{\Lie}=0\subseteq{\gamma_2^{\Lie}(\Lieh)},~~$ $\big[{{\ze}_{1}^{\Lie}({\Lieh})},~\Lieh\big]_{\Lie}=0\subseteq{\gamma_3^{\Lie}(\Lieh)}~~$ and $~\big[{{\ze}_{n-1}^{\Lie}({\Lieh})},~\Lieh\big]_{\Lie}=span\{a_4\}\subseteq{\gamma_{n+1}^{\Lie}(\Lieh)}$ for all $n>2.$
\end{Ex}

\begin{Rem}
Lemma \ref{Lem1} {\it b)} in case $n=1$ provides the isoclinism given in Proposition 3.20 {\it c)} in \cite{BC}.
\end{Rem}

\begin{Co}
Let $\Lieg$ be a Leibniz algebra. Then $\Lieg$ is $n$-\Lie-isoclinic to some Leibniz algebra $\Lieh$ satisfying ${Z}_{\Lie}(\Lieh)\cap \gamma_n^{\Lie}(\Lieh)\subseteq \gamma_{n+1}^{\Lie}(\Lieh). $
\end{Co}
\begin{proof}
Let $\gamma=( \gamma_{n}^{\Lie}(\Lieg)\setminus \gamma_{n+1}^{\Lie}(\Lieg)).$  Then   $\Liet_n:=\gamma\cap {Z}_{\Lie}(\Lieg )$ is a two-sided ideal of $\Lieg$ satisfying $\Liet_n\cap \gamma_{n+1}^{\Lie}(\Lieg)=0.$
Therefore ${\Lieg}/{\Liet_n}$ is $n$-\Lie-isoclinic to $\Lieg$ by Lemma \ref{Lem1} {\it d)}.  Take $\Lieh:=\Lieg/\Liet_n.$ It remains to show that ${Z}_{\Lie}(\Lieh)\cap \gamma_n^{\Lie}(\Lieh)\subseteq \gamma_{n+1}^{\Lie}(\Lieh).$ Indeed, let $x+\Liet_n\in {Z}_{\Lie}(\Lieh)\cap \gamma_n^{\Lie}(\Lieh)=\frac{{Z}_{\Lie}(\Lieg)}{\Liet_n}\cap\frac{\gamma_n^{\Lie}(\Lieg)}{\Liet_n}.$ Without loss of generality, assume that $x\in\gamma_n^{\Lie}(\Lieg).$ Then $y:=[x,g]_{lie}\in\Liet_n$ for all $g\in\Lieg,$ and $y\in \gamma_{n+1}^{\Lie}(\Lieg).$  So $y=0$ since $\Liet_n\cap \gamma_{n+1}^{\Lie}(\Lieg)=0.$ This implies that $x\in{Z}_{\Lie}(\Lieg), $ and so $x\in \gamma_n^{\Lie}(\Lieg)\cap {Z}_{\Lie}(\Lieg)=\gamma_{n+1}^{\Lie}(\Lieg)\cap {Z}_{\Lie}(\Lieg)+\Liet_n,$ and thus $x+\Liet_n\in \gamma_{n+1}^{\Lie}(\Lieh)$ since $\gamma_{n+1}^{\Lie}(\Lieh)=\frac{\gamma_{n+1}^{\Lie}(\Lieg)}{\Liet_n}.$
\end{proof}

\begin{Pro} \label{backward}
Let $(\eta, \xi): (g_1) \underset{n}\sim (g_2)$ be a $n$-$\Lie$-isoclinism and consider the following set  $$\Lieg_2^{\eta} = \left\{ (g,x+{{\ze}_{n}^{\Lie}({\Lieg_1})}) \in  \Lieg_2 \times  \frac{\Lieg_1}{{{\ze}_{n}^{\Lie}({\Lieg_1})}} ~\mid ~g+{{\ze}_{n}^{\Lie}({\Lieg_2})} = \eta(x+{{\ze}_{n}^{\Lie}({\Lieg_1})}) \right\}.$$ There is a $n$-{\Lie}-central extension $(\eta^{\ast}(g_2)): 0 \to {{\ze}_{n}^{\Lie}({\Lieg_2})} \to \Lieg_2^{\eta} \stackrel{\overline{\pi}_2}\to \frac{\Lieg_1}{{{\ze}_{n}^{\Lie}({\Lieg_1})}}\to 0$ isomorphic to $(g_2)$, and $n$-\Lie-isoclinic to $(g_1).$
\end{Pro}
\begin{proof}
Choose $\eta': id_{\Lieg_1/{{\ze}_{n}^{\Lie}({\Lieg_1})}}$ the identity map, and define
$\xi':\gamma_{n+1}^{\Lie}(\Lieg_1)\to \gamma_{n+1}^{\Lie}(\Lieg_2^{\eta})$ by $\xi'(C_1^{n+1}(\bar{x}_1,\ldots,\bar{x}_{n+1}))=C_2^{\eta}(\bar{x}_1,\ldots,\bar{x}_{n+1})$ where $$C_2^{\eta}(\bar{x}_1,\ldots,\bar{x}_{n+1})=(\xi(C_1^{n+1}(\bar{x}_1,\ldots,\bar{x}_{n+1})),[[[\bar{x}_1,\bar{x}_2]_{lie},\ldots, \bar{x}_n]_{lie},\bar{x}_{n+1}]_{lie}).$$ Then $\xi'$ is clearly an isomorphism since $\xi:\gamma_{n+1}^{\Lie}(\Lieg_1)\to \gamma_{n+1}^{\Lie}(\Lieg_2)$ is an isomorphism, and the diagram
\begin{equation}
 \xymatrix{
\frac{\Lieg_1}{{\ze}_{n}^{\Lie}({\Lieg_1})}  \times\ldots\times \frac{\Lieg_1}{{\ze}_{n}^{\Lie}({\Lieg_1})}   \ar[r]^-*+{C_1^{n+1}} \ar[d]_{\eta^{ n+1}}  & \gamma_{n+1}^{\Lie}(\Lieg_1)  \ar[d]^{\xi'_n}\\
\frac{\Lieg_1}{{\ze}_{n}^{\Lie}({\Lieg_1})}  \times\ldots\times \frac{\Lieg_1}{{\ze}_{n}^{\Lie}({\Lieg_1})}    \ar[r]^-*+{C_2^{\eta}} & \gamma_{n+1}^{\Lie}(\Lieg_2^{\eta}) }
\end{equation}
is commutative by construction.
\end{proof}

\begin{Pro}\label{xieta}
For a $n$-$\Lie$-isoclinism $(\eta, \xi): (g_1) \underset{n}\sim (g_2)$, the following statements hold:
\begin{enumerate}
\item[a)] $\xi(g)+{{\ze}_{n}^{\Lie}({\Lieg_2})}  = \eta ( g+{{\ze}_{n}^{\Lie}({\Lieg_1})} )$, for all $g \in \gamma_{n+1}^{\Lie}(\Lieg_1)  $.
\item[b)] $\xi \left( {{\ze}_{n}^{\Lie}({\Lieg_1})} \cap \gamma_{n+1}^{\Lie}(\Lieg_1)  \right) = {{\ze}_{n}^{\Lie}({\Lieg_2})}  \cap \gamma_{n+1}^{\Lie}(\Lieg_2)  $.
\item[c)]$\xi(C_1^{n+1}(\bar{x}_1,\ldots,\bar{x}_{n},\bar{g})) = C_2^{n+1}(\widetilde{\xi(x_1)},\ldots,\widetilde{\xi(x_n)},\tilde{h})$, for all $g \in \Lieg_1,$ $h \in \Lieg_2$ with     $\eta(\bar{g})=\tilde{h},$ and  $x_i \in \Lieg_1, i=1, \dots, n,$    where $\bar{g}=g+{{\ze}_{n}^{\Lie}({\Lieg_1})},$ $\bar{x_i}=x_i+{{\ze}_{n}^{\Lie}({\Lieg_1})}$, $\widetilde{\xi(x_i)} = \xi(x_i) + {{\ze}_{n}^{\Lie}({\Lieg_2})}, i=1, \dots, n,$  and $\tilde{h}= h+{{\ze}_{n}^{\Lie}({\Lieg_2})}.$
    \end{enumerate}
\end{Pro}
\begin{proof}

{\it a)}
Let $g\in \gamma_{n+1}^{\Lie}(\Lieg_1).$
 By Proposition \ref{backward}, $(id_{\Lieg_1/{{\ze}_{n}^{\Lie}({\Lieg_1})}}, \xi') : (g_1) \underset{n}\sim (\eta^{\ast} g_2) $ is a $n$-\Lie-isoclinism  where $(\xi(g),g+{{\ze}_{n}^{\Lie}({\Lieg_1})})\in \gamma_{n+1}^{\Lie}(\Lieg_2^{\eta})$  for all $g\in\gamma_{n+1}^{\Lie}(\Lieg_1)\subseteq \Lieg_2^{\eta}.$ The result follows by definition of $\Lieg_2^{\eta}.$
\medskip

{\it b)}  Let $g \in{{\ze}_{n}^{\Lie}({\Lieg_1})} \cap\gamma_{n+1}^{\Lie}(\Lieg_1).$ Then $\xi(g)\in\gamma_{n+1}^{\Lie}({\Lieg_2}).$ Assume that $\xi(g)\notin{{\ze}_{n}^{\Lie}({\Lieg_2})} .$ Then $\xi(g)+{{\ze}_{n}^{\Lie}({\Lieg_2})} \neq {{\ze}_{n}^{\Lie}({\Lieg_2})} .$ So by a),  $\eta ( g+{{\ze}_{n}^{\Lie}({\Lieg_1})} )\neq {{\ze}_{n}^{\Lie}({\Lieg_2})} ,$ i.e. $g+{{\ze}_{n}^{\Lie}({\Lieg_1})} \notin \Ker(\eta)=0$ i.e. $g\notin {{\ze}_{n}^{\Lie}({\Lieg_1})}.$ A contradiction. So $\xi(g)\in {{\ze}_{n}^{\Lie}({\Lieg_2})}  \cap \gamma_{n+1}^{\Lie}(\Lieg_2).$

Conversely, if $h\in{{\ze}_{n}^{\Lie}({\Lieg_2})} \cap\gamma_{n+1}^{\Lie}(\Lieg_2),$ then   $h=\xi(g)$ for some $g\in \gamma_{n+1}^{\Lie}(\Lieg_1)$ since $\xi$ is onto. It follows by {\it a)} that $\eta(g+{{\ze}_{n}^{\Lie}({\Lieg_1})})=\xi(g)+{{\ze}_{n}^{\Lie}({\Lieg_2})}=h+{{\ze}_{n}^{\Lie}({\Lieg_2})}={{\ze}_{n}^{\Lie}({\Lieg_2})},$ which implies that $g\in {{\ze}_{n}^{\Lie}({\Lieg_1})}$ since $\eta$ is one-to-one, and thus $g\in{{\ze}_{n}^{\Lie}({\Lieg_1})} \cap\gamma_{n+1}^{\Lie}(\Lieg_1).$
\medskip

 {\it c)}  Since $(\eta, \xi)$ is a $n$-{\Lie}-isoclinism, then the commutativity of diagram (\ref{square isoclinic}) provides:
\[
\begin{aligned}
\xi(C_1^{n+1}(\bar{x}_1,\ldots,\bar{x}_{n},\bar{g})) &=C_2^{n+1}(\eta^{n+1}(\bar{x}_1,\ldots,\bar{x}_n,\bar{g}))\\&=C_2^{n+1}(\eta(\bar{x}_1),\ldots,\eta(\bar{x}_n),\eta(\bar{g}))\\&=C_2^{n+1}(\widetilde{\xi(x_1)},\ldots,\widetilde{\xi(x_n)},\tilde{h}).
\end{aligned}
\]
\end{proof}
\bigskip

For a $n$-\Lie-isoclinism  $(\eta, \xi): (g_1) \underset{n}\sim (g_2),$ consider the following sets:

$$ \LieK=\{(g,h)\in \Lieg_1\oplus\Lieg_2~/~\eta(g+{{\ze}_{n}^{\Lie}({\Lieg_1})})=h+{{\ze}_{n}^{\Lie}({\Lieg_2})}\},$$ $$Z_{\Lieg_1}=\{(g,0)~/~g\in{\ze}_{n}^{\Lie}(\Lieg_1)\}~~~~~~~\mbox{and}~~~~~~~~Z_{\Lieg_2}=\{(0,h)~/~h\in{\ze}_{n}^{\Lie}(\Lieg_2)\}.$$

\begin{Le}\label{iso}
The following assertions are true:
\begin{enumerate}
\item[a)] $\LieK$ is a subalgebra of $\Lieg_1\times\Lieg_2$.

\item[b)]  $\gamma_{n+1}^{\Lie}(\LieK)=\{(g,\xi(g))~|~g\in\gamma_{n+1}^{\Lie}( \Lieg_1)\}$.

\item[c)]  $Z_{\Lieg_i},~i=1,2$  are two-sided ideals of $\LieK$ satisfying $Z_{\Lieg_i}\cap\gamma_{n+1}^{\Lie}(\LieK)=0$.

\item [d)]  $\Lieg_i\cong {\LieK}/{Z_{\Lieg_j}}\underset{n}\sim\LieK,~~i,j=1,2$ with $i\neq j;$

\item [e)]  $ {\LieK}/{Z_{\Lieg_j}} \underset{n}\sim {\LieK}/{Z_{\Lieg_i}}\oplus {\LieK}/\gamma_{n+1}^{\Lie}(\LieK)$, provided that  $[{\ze}_{n-1}^{\Lie}({\Lieg_i}), {\Lieg}_i]_{\Lie} \subseteq \gamma_{n+1}^{\Lie}({\Lieg}_i)$
for $~i=1,2$.
\end{enumerate}
\end{Le}
\begin{proof}
 The proof of {\it a)}  is  straightforward.
 {\it b)} follows from the property {\it a)} of Proposition \ref{xieta}.
\medskip

  For {\it c)}, it is easy to check that $Z_{\Lieg_i}, i=1, 2,$  are two-sided ideals of ${\LieK}$ since ${\ze}_{n}^{\Lie}(\Lieg_i)$   are two-sided ideals of $\Lieg_i, i=1, 2.$  That $Z_{\Lieg_i}\cap\gamma_{n+1}^{\Lie}(\LieK)=0,$ $i=1,2,$ is due to $\xi$ being one-to-one.
\medskip

  To prove  {\it d)}, it is easy to check that the maps $\tau_i: \LieK\to \Lieg_i$ defined by $\tau_i(g_1,g_2)=g_i$ are surjective homomorphisms with kernel $\Ker(\tau_i)=Z_{\Lieg_j},$ $i,j=1,2$ with $i\neq j.$ The isomorphisms follow by the first isomorphism theorem.
  The isoclinisms follow from  Lemma \ref{Lem1} {\it d)} since $Z_{\Lieg_i}\cap\gamma_{n+1}^{\Lie}(\LieK)=0,~i=1,2$, by statement {\it c)}.
\medskip

  To prove {\it e)}, consider the sets $\LieK_{\Lieg_i}=\{(k+Z_{\Lieg_i},k+\gamma_{n+1}^{\Lie}(\LieK))~/~k\in\LieK\},~i=1,2.$ Then it is clear that $\LieK_{\Lieg_i}$ are subalgebras of $\LieK/Z_{\Lieg_i}\oplus \LieK/\gamma_{n+1}^{\Lie}(\LieK)$ and the mappings $\alpha_i: \LieK\to\LieK_{\Lieg_i}$ defined by $\alpha_i(k)=(k+Z_{\Lieg_i},k+\gamma_{n+1}^{\Lie}(\LieK))$ are surjective homomorphisms by definition, and one-to-one due to  $Z_{\Lieg_i}\cap\gamma_{n+1}^{\Lie}(\LieK)=0.$ So $\LieK$ is isomorphic to $\LieK_{\Lieg_i},~i=1,2.$

   Now  apply Lemma \ref{Lem1} {\it b)} with $\Lieg:=\LieK/Z_{\Lieg_i}\oplus \LieK/\gamma_{n+1}^{\Lie}(\LieK),$ and $\Lieh:=\LieK_{\Lieg_i},$ then statement {\it  d)} concludes the proof.

  Keep in mind that Lemma \ref{Lem1} {\it b)} is well applied since the hypotheses imply the requirements of  Lemma \ref{Lem1} {\it b)}.
 Indeed, for the case $i=1$,  assume that $[{\ze}_{n-1}^{\Lie}({\Lieg_1}), {\Lieg_1}]_{\Lie}\subseteq \gamma_{n+1}^{\Lie}(\Lieg_1),$ and let $k=(g,h)\in \LieK$ such that
          $(k+Z_{\Lieg_1},k+\gamma_{n+1}^{\Lie}(\LieK))\in {\ze}_{n-1}^{\Lie}(\LieK_{\Lieg_1}).$ This implies that for every $g_1,\ldots,g_{n-1}\in \Lieg_1,$ and appropriate $h_1,\ldots,h_{n-1}\in \Lieg_2$  so that $k_1,\ldots,k_{n-1}\in\LieK $ with $k_i=(g_i,h_i),$ we have $$[[[k,k_1]_{lie},k_2]_{lie},\ldots,k_{n-1}]_{lie}\in Z_{\Lieg_1}\cap\gamma_{n+1}^{\Lie}(\LieK)=0,$$ and thus $$([[[g,g_1]_{lie},g_2]_{lie},\ldots,g_{n-1}]_{lie},[[[h,h_1]_{lie},h_2]_{lie},\ldots,h_{n-1}]_{lie})=0.$$  So $[[[g,g_1]_{lie},g_2]_{lie},\ldots,g_{n-1}]_{lie}=0$ which implies that  $g\in{{\ze}_{n-1}^{\Lie}({\Lieg_1})}.$
          Therefore, for all $k'=(g',h')\in \LieK,$  $[(k+Z_{\Lieg_1},k+\gamma_{n+1}^{\Lie}(\LieK)),(k'+Z_{\Lieg_1},k'+\gamma_{n+1}^{\Lie}(\LieK))]_{lie}= (([g,g']_{lie},[h,h']_{lie})+Z_{\Lieg_1},([g,g']_{lie},[h,h']_{lie})+\gamma_{n+1}^{\Lie}(\LieK))\in \gamma_{n+1}^{\Lie}(\LieK_{\Lieg_1})$ since $[g,g']_{lie}\in [{\ze}_{n-1}^{\Lie}({\Lieg_1}), {\Lieg_1}]_{\Lie} \subseteq\gamma_{n+1}^{\Lie}(\Lieg_1).$

          The proof for $i=2$ is similar since $\eta$ and $\xi$ are isomorphisms.
\end{proof}
\bigskip

As a consequence of Lemma \ref{iso},  we have the following  characterization of $n$-\Lie-isoclinism classes of Leibniz algebras.

\begin{Co} \label{isoclinic pair}
Given $\Lieg_1$ and $\Lieg_2, $ two Leibniz algebras.  Then $\Lieg_1$ and $\Lieg_2 $ are  $n$-$\Lie$-isoclinic if and only if
there exist a Leibniz algebra $\Lieh$ that  is $n$-\Lie-isoclinic  to $\Lieg_1,$ and a surjective homomorphism $\theta$ from $\Lieh$ onto $\Lieg_2$  such that $\Ker(\theta)\cap\gamma_{n+1}^{\Lie}(\Lieh)=0.$
\end{Co}
\begin{proof}
Assume that  $\Lieg_1$ and $\Lieg_2 $ are  $n$-$\Lie$-isoclinic and take $\Lieh=\LieK.$ Then by the proof of Lemma \ref{iso}, we have $\Lieh\underset{n}\sim\Lieg_1$ and the surjective homomorphism  $\theta:=\tau_2: \LieK\to \Lieg_2$ defined by $\tau_2(g_1,g_2)=g_2, $ satisfying $\Ker(\theta)\cap\gamma_{n+1}^{\Lie}(\Lieh)=Z_{\Lieg_1}\cap\gamma_{n+1}^{\Lie}(\LieK)=0.$

Conversely, let $\Lieh$ be a Leibniz algebra such that $\Lieh\underset{n}\sim\Lieg_1$ and  $\theta:\Lieh\to\Lieg_2$ be a surjective homomorphism satisfying $\Ker(\theta)\cap\gamma_{n+1}^{\Lie}(\Lieh)=0.$  Then by the property {\it d)} of Lemma \ref{Lem1},  $\Lieg_1\underset{n}\sim\Lieh\underset{n}\sim\frac{\Lieh}{\Ker(\theta)}\cong\Lieg_2.$
\end{proof}

\begin{Co}\label{n-1}
Given $\Lieg_1$ and $\Lieg_2, $ two Leibniz algebras such that $\Lieg_2$ is isomorphic to $\Lieg_1/\LieI$ for some two-sided ideal $\LieI$ of $\Lieg_1$ satisfying $\LieI\cap \gamma_{n+1}^{\Lie}(\Lieg_1)=0,$ then the following hold:
\begin{enumerate}
\item[a)]  $\frac{\Lieg_1}{{\ze}_{n}^{\Lie}({\Lieg_1})}$ is $(n$-$k)$-\Lie-isoclinic  to $\frac{\Lieg_2}{{\ze}_{n}^{\Lie}({\Lieg_2})},~$ for $k=0,\ldots, n.$

\item[b)]  $\gamma_{k+1}^{\Lie}(\Lieg_1)$  is $(n$-$k)$-\Lie-isoclinic  to $\gamma_{k+1}^{\Lie}(\Lieg_2),~$ for $k=0,\ldots, n.$

\item[c)] $\Lieg_1$  is $m$-$\Lie$-isoclinic to $\Lieg_2, ~$ for all $m\geq n.$
\end{enumerate}
\end{Co}
\begin{proof}
 To prove {\it a)}, let  $\Lies$ be a two-sided ideal of $\Lieg_1$  such that ${\ze}^{\Lie}_{k}(\Lieg_1/\LieI)=\Lies/\LieI.$  One easily verifies that ${\ze}^{\Lie}_{k}(\Lieg_1)\subseteq \Lies.$
Now let  $g\in \Lies\cap \gamma_{n+1-k}^{\Lie}(\Lieg_1).$ Then for all $h_1,h_2,\ldots,h_k\in\Lieg_1,$ we have $[[[g,h_1]_{lie},h_2]_{lie},\ldots,h_k]_{lie}\in \LieI\cap \gamma_{n+1}^{\Lie}(\Lieg_1)=0.$ So $g\in{\ze}^{\Lie}_{k}(\Lieg_1)$ and thus $\Lies\cap \gamma_{n+1-k}^{\Lie}(\Lieg_1)\subset {\ze}^{\Lie}_{k}(\Lieg_1).$ This implies that $\frac{\Lies}{{\ze}^{\Lie}_{k}(\Lieg_1)}\cap\gamma_{n+1-k}^{\Lie}(\frac{\Lieg_1}{{\ze}^{\Lie}_{k}(\Lieg_1)})= 0.$ It follows by the property {\it d)} of Lemma \ref{Lem1} that $\frac{\Lieg_1}{{\ze}^{\Lie}_{k}(\Lieg_1)}$ is  $(n-k)$-{\Lie}-isoclinic to $\frac{\Lieg_1}{{\ze}^{\Lie}_{k}(\Lieg_1)}/\frac{\Lies}{{\ze}^{\Lie}_{k}(\Lieg_1)}
\cong\Lieg_1/\Lies\cong \frac{\Lieg_1}{\LieI}/\frac{\Lies}{\LieI}\cong \frac{\Lieg_1}{\LieI}/{\ze}^{\Lie}_{n}(\frac{\Lieg_1}{\LieI})\cong \frac{\Lieg_2}{{\ze}^{\Lie}_{k}(\Lieg_2)}.$
\medskip

{\it b)} Let $k\in\{0,\ldots, n\}$ and set $\Liet:=\LieI\cap \gamma_{k+1}^{\Lie}(\Lieg_1).$ It is not hard to verify that $\Liet\cap\gamma_{n+1-k}^{\Lie}(\Lieg_1)\subseteq \LieI\cap\gamma_{n+1}^{\Lie}(\Lieg_1)=0.$ So again by the property {\it d)} of Lemma \ref{Lem1}, we have  $\gamma_{k+1}^{\Lie}(\Lieg_1)$  is $(n$-$k)$-\Lie-isoclinic  to $\frac{\gamma_{k+1}^{\Lie}(\Lieg_1)}{\Liet}\cong \frac{\gamma_{k+1}^{\Lie}(\Lieg_1)+\LieI}{\LieI}\cong \gamma_{k+1}^{\Lie}(\frac{\Lieg_1}{\LieI})\cong \gamma_{k+1}^{\Lie}(\Lieg_2).$
\medskip

For {\it c)}, the $m$-$\Lie$-isoclinism $\Lieg_1\underset{m}\sim\frac{\Lieg_1}{\LieI}\cong\Lieg_2$ is obtained also by the property {\it d)} of Lemma \ref{Lem1}, since $\LieI\cap \gamma_{m+1}^{\Lie}(\Lieg_1)\subseteq \LieI\cap \gamma_{n+1}^{\Lie}(\Lieg_1)=0.$
\end{proof}

\begin{Rem}
If $\Lieg_1$ and $\Lieg_2$ are $n$-{\Lie}-isoclinic Leibniz algebras, then there exists the two-sided ideal ${\cal J}$ satisfying the requirements of Corollary \ref{n-1} thanks to Corollary \ref{isoclinic pair}. Other broad class of algebras satisfying the requirements of Corollary \ref{n-1} are  {\Lie}-nilpotent Leibniz algebras of class $n$.

Another example of non {\Lie}-nilpotent Leibniz algebra satisfying the requirements of Corollary \ref{n-1} is the three-dimensional Leibniz algebra  $\Lieg_1$ with basis $\{a_1,a_2,a_3 \}$, with bracket operation $[a_1,a_3]=a_1$ (see algebra 2 d) in the classification given in \cite{CILL}). Take the two-sided ideal ${\cal J}=  \langle \{a_2\} \rangle$, $\gamma_{n+1}^{\Lie} (\Lieg_1) = \langle \{a_1 \} \rangle$, hence the intersection is zero. Take $\Lieg_2= \Lieg_1/{\cal J} = \langle \{ a_1, a_3 \} \rangle$.
\end{Rem}

\begin{Pro}\label{Prop1}
Given $\Lieg_1$ and $\Lieg_2, $ two $n$-$\Lie$-isoclinic Leibniz algebras. Then there exist two Leibniz algebras $\Lieh_1$ and $\Lieh_2$ and a {\Lie}-nilpotent Leibniz algebra $\Lien$ of class at least $n$ satisfying the following:
 \begin{enumerate}
\item[a)]  $\Lieg_1$ is $n$-\Lie-isoclinic to $\Lieh_1\oplus\Lien$.

\item[b)]  $\Lieg_2$  is $(n$-$1)$-\Lie-isoclinic  to $\Lieh_2$.

\item[c)]  $\Lieh_2$ is $n$-\Lie-isoclinic  to $\Lieh_2+{{\ze}_{n}^{\Lie}({\Lieh_1\oplus\Lien})}$, provided that $[Z_{n-1}^{\Lie}({\Lieg}_2), {\Lieg}_2]_{\Lie} \subseteq \gamma_{n+1}^{\Lie}({\Lieg}_2)$.

\item[d)]   $\Lieh_1\oplus\Lien=\Lieh_2+{{\ze}_{n}^{\Lie}({\Lieh_1\oplus\Lien})}$ if $\Lieg_1$ and $\Lieg_2 $ are finite dimensional  and $[Z_{n-1}^{\Lie}({\Lieg}_2), {\Lieg}_2]_{\Lie}$ $\subseteq \gamma_{n+1}^{\Lie}({\Lieg}_2)$.
\end{enumerate}
\end{Pro}
\begin{proof}  Keeping in mind the above notations,  consider the set ${\LieH}=\{\big((g,0)+Z_{\Lieg_2}, (g,0)+\gamma_{n+1}^{\Lie}({\LieK})\big)/~g\in{\ze}_{n}^{\Lie}({\Lieg_1})\cap \gamma_{n}^{\Lie}(\Lieg_1)\}.$ It is easy to check that ${\LieH}$ is a two-sided ideal of ${\LieL}:={\LieK}/Z_{{\Lieg}_2}\oplus \LieK/\gamma_{n+1}^{\Lie}(\LieK).$ 

Now define the mappings $\alpha_1, \alpha_2:\LieK\to \LieL/\LieH$ respectively by
$$\alpha_1(t)=\big(t+Z_{\Lieg_2}, \gamma_{n+1}^{\Lie}(\LieK)\big)+\LieH~~\mbox{ and}~~\alpha_2(t)=\big(t+Z_{\Lieg_2}, t+\gamma_{n+1}^{\Lie}(\LieK)\big)+\LieH.$$
It is easy to check that $\alpha_1$ and $\alpha_2$ are Leibniz algebra homomorphisms with  $\Ker(\alpha_1)=Z_{\Lieg_2}$ and $\Ker(\alpha_2)=Z_{\Lieg_1}\cap \gamma_{n}^{\Lie}(\LieK)$. Now set $\Lieh_1:=\alpha_1(\LieK),~$ $\Lieh_2:=\alpha_2(\LieK)~$ and   $~\Lien=\{(Z_{\Lieg_2},t+\gamma_{n+1}^{\Lie}(\LieK))+\LieH~|~t\in\LieK\}.$
 Then one easily verifies that  ${\Lien}$ is  {\Lie}-nilpotent  of class at least $n$  since  $\gamma_{j}^{\Lie}(\Lien)=\{(Z_{\Lieg_2},\gamma_{n+1}^{\Lie}(\LieK))\}$ iff $\gamma_{j}^{\Lie}(\LieK)\subseteq\gamma_{n+1}^{\Lie}(\LieK),$ iff $j\geq n+1.$ Now combining the properties  {\it a)} and {\it d)}  of Lemma \ref{Lem1}  and Lemma \ref{iso} {\it d)} with the first isomorphism theorem, we have
 $${\Lieh}_1\oplus {\Lien}\underset{n}\sim {\Lieh}_1=\alpha_1({\LieK})\cong\frac{{\LieK}}{\Ker(\alpha_1)}\cong\frac{\LieK}{Z_{\Lieg_2}}\cong\Lieg_1,$$
   and
   $${\Lieh}_2=\alpha_2(\LieK)\cong\frac{\LieK}{\Ker(\alpha_2)}=\frac{\LieK}{Z_{\Lieg_1}\cap \gamma_{n}^{\Lie}(\LieK)}\underset{n-1}\sim\frac{\LieK}{Z_{\Lieg_1}}\cong\Lieg_2.$$
     This proves {\it a)} and {\it b)}.
\medskip

 The result  {\it c)} is due to    the property  {\it b)}  of Lemma \ref{Lem1} since $\Lieh_2$ is a subalgebra of ${\Lieh}_1\oplus{\Lien}$ and the
 condition $[{\ze}_{n-1}^{\Lie}(\frak{h}_2), {\frak h}_2]_{\Lie} \subseteq \gamma_{n+1}^{\Lie}(\frak{h}_2)$ holds.
Indeed, for $k=(x,y), k'=(x',y')\in {\LieK}$ such that $\alpha_1(k)\in {\ze}_{n-1}^{\Lie}(\frak{h}_2)$, then $y\in {\ze}_{n-1}^{\Lie}(\Lieg_2)$, since for $y_1,\ldots ,y_{n-1}\in\Lieg_2,$ and appropriate $x_1,\ldots ,x_{n-1}\in\Lieg_1,$  such that $k_i=(x_i,y_i)\in {\LieK}, i = 1, \dots, n-1$,  and keeping in mind that $k\in {\ze}_{n-1}^{\Lie}(\frak{h}_2),$  we have $\big[[(k+Z_{\Lieg_2},k+\gamma_{n+1}^{\Lie}(\LieK)),(k_1+Z_{\Lieg_2},k_1+\gamma_{n+1}^{\Lie}(\LieK))]_{lie},\dots,(k_{n-1}+Z_{\Lieg_2},k_{n-1}+\gamma_{n+1}^{\Lie}(\LieK))\big]_{lie}\in \LieH,$ i.e.   $[[[k,k_1]_{lie},\ldots,k_{n-1}]_{lie}-(g,0)\in  Z_{\Lieg_2}\cap\gamma_{n+1}^{\Lie}(\LieK)=0$  for some $g\in{\ze}_{n}^{\Lie}({\Lieg_1})\cap \gamma_{n}^{\Lie}(\Lieg_1).$ This implies that $[[[x,x_1]_{lie},x_2]_{lie},\ldots,x_{n-1}]_{lie}=g$ and $[[[y,y_1]_{lie},y_2]_{lie},\ldots,y_{n-1}]_{lie}=0,$ and thus $y\in {\ze}_{n-1}^{\Lie}(\Lieg_2).$

Now we have  $[(k+Z_{\Lieg_2},k+\gamma_{n+1}^{\Lie}(\LieK)),(k'+Z_{\Lieg_2},k'+\gamma_{n+1}^{\Lie}(\LieK))]_{lie}+\LieH= (([x,x']_{lie},$ $[x,x']_{lie})+Z_{\Lieg_2},([x,x']_{lie},[y,y']_{lie})+\gamma_{n+1}^{\Lie}(\LieK))+\LieH \in \gamma_{n+1}^{\Lie}(\Lieh_2)$  since $[y,y']_{lie}\in [{\ze}_{n-1}^{\Lie}({\Lieg_2}), {\Lieg_2}]_{\Lie} \subseteq\gamma_{n+1}^{\Lie}(\Lieg_2).$
\medskip

To prove  {\it d)}, we have by {\it  b)} that ${\Lieg_2} \underset{n-1}\sim {\Lieh}_2.$ So by the property {\it c)} of Corollary \ref{n-1}, $\Lieg_2\underset{n}\sim\Lieh_2.$
 So $\Lieh_1\oplus\Lien\underset{n}\sim\Lieg_1\underset{n}\sim\Lieg_2\underset{n}\sim\Lieh_2.$ The result now follows from  the property  {\it c)}  of Lemma \ref{Lem1} since ${\LieK}$ is finite dimensional.  Keep in mind that $[{\ze}_{n-1}^{\Lie}(\frak{h}_2), {\frak h}_2]_{\Lie} \subseteq \gamma_{n+1}^{\Lie}(\frak{h}_2)$ is showed in the proof of statement {\it c)}.
\end{proof}

\begin{De} \label{isoclinic1}
A homomorphism of $n$-$\Lie$-central extensions $(\alpha, \beta, \gamma) : (g_1) \to (g_2)$ is said to be $n$-$\Lie$-isoclinic, if there exists an isomorphism $\beta_{ n+1}' :  \gamma_{n+1}^{\Lie}(\Lieg_1) \to \gamma_{n+1}^{\Lie}(\Lieg_2)$ with $(\gamma, \beta_{n+1}') : (g_1) \underset{n}\sim (g_2)$. The mapping $\beta$ is referred to as a  $n$-$\Lie$-isoclinic homomorphism.

If $\beta$ is in addition an epimorphism (resp., monomorphism), then  $(\alpha, \beta, \gamma)$ is called a $n$-\Lie-isoclinic epimorphism (resp., monomorphism).
\end{De}

\begin{Pro} \label{equivalence}
A homomorphism of $n$-{\Lie}-central extensions $(\alpha, \beta, \gamma) : (g_1) \to (g_2)$ is  $n$-$\Lie$-isoclinic if and only if $\gamma$ is an isomorphism and $\Ker(\beta) \cap \gamma_{n+1}^{\Lie}(\Lieg_1) =0$.
\end{Pro}
\begin{proof}  Suppose that $(\alpha, \beta, \gamma) : (g_1) \to (g_2)$ is a $n$-$\Lie$-isoclinic homomorphism. Then $(\gamma, \beta') : (g_1) \underset{n}\sim (g_2)$ is a  $n$-$\Lie$-isoclinism for some isomorphism $\beta' : \gamma_{n+1}^{\Lie}(\Lieg_1) \to \gamma_{n+1}^{\Lie}(\Lieg_2),$ and   $\gamma: \frac{\Lieg_1}{{\ze}_{n}^{\Lie}({\Lieg_1})} \to \frac{\Lieg_2}{{\ze}_{n}^{\Lie}({\Lieg_2})}$ is an isomorphism. Let $g\in \Ker(\beta) \cap  \gamma_{n+1}^{\Lie}(\Lieg_1).$ Then $\beta(g)=0$ and $g=C_1^{n+1}(\bar{x}_1,\ldots,\bar{x}_{n+1})$  with $\bar{x}_i=x_i+{\ze}_{n}^{\Lie}({\Lieg_1}),$ for some $x_1,\ldots, x_{n+1}\in \Lieg_1.$
Since $(\gamma, \beta') : (g_1) \underset{n}\sim (g_2)$ is a  $n$-$\Lie$-isoclinism, we have
$$
\begin{aligned}
\beta'(g)&=\beta'(C_1^{n+1}(\bar{x}_1,\ldots,\bar{x}_{n+1}))\\&=C_2^{n+1}\big(\gamma^{n+1}(\bar{x}_1,\ldots,\bar{x}_{n+1})\big)\\&=C_2^{n+1}\big(\gamma(\bar{x}_1),\ldots,\gamma(\bar{x}_{n+1})\big)\\&=C_2^{n+1}\big(\widetilde{\beta({x}_1)}, \ldots,\widetilde{\beta({x}_{n+1})}\big)\\&=\beta(C_1^{n+1}(\bar{x}_1,\ldots,\bar{x}_{n+1}))=\beta(g)=0.
\end{aligned}$$
 Therefore $g=0$ since $\beta'$ is one-to-one.

Conversely, suppose that  $\Ker(\beta) \cap \gamma_{n+1}^{\Lie}(\Lieg_1) =0,$ then define $\beta':\gamma_{n+1}^{\Lie}(\Lieg_1)\to \gamma_{n+1}^{\Lie}(\Lieg_2)$ by $\beta'(g)=\beta(g).$ Clearly $\beta'$ is  one-to-one due to $\Ker(\beta) \cap \gamma_{n+1}^{\Lie}(\Lieg_1) =0.$ To show that $\beta'$ is onto, let $h\in  \gamma_n^{\Lie}(\Lieg_2).$ Then $h=[[[y_1,y_2]_{lie},y_3]_{lie},\ldots,y_{n+1}]_{lie}=C_2^{n+1}(\tilde{y}_1,\ldots,\tilde{y}_{n+1})$ for some $y_1,\ldots, y_{n+1}\in\Lieg_2.$ Since  $\gamma_1$ is onto, we have $\tilde{y_i}=\gamma(\bar{x}_i)=\widetilde{\beta(x_i)},$ $~i=1,\ldots,n+1.$ It follows that
\[\begin{aligned}
\beta'(C_1^{n+1}(\bar{x}_1,\ldots,\bar{x}_{n+1}))&=\beta(C_1^{n+1}(\bar{x}_1,\ldots,\bar{x}_{n+1}))\\&=C_2^{n+1}\big(\widetilde{\beta({x}_1)}, \ldots,\widetilde{\beta({x}_{n+1})}\big)\\&=C_2^{n+1}(\tilde{y}_1,\ldots,\tilde{y}_{n+1})=y.
\end{aligned}\]
\end{proof}

\begin{Rem}
From the proof of Proposition \ref{equivalence}, it follows that the isomorphism $\beta_{n+1}'$  in Definition \ref{isoclinic1} is the restriction of $\beta$ to $\gamma_{n+1}^{\Lie}(\Lieg_1).$
\end{Rem}

\begin{Pro} \label{character}
Let $\beta : \Lieg_1 \to \Lieg_2$ be a homomorphism of Leibniz algebras. Then $\beta$ induces a $n$-$\Lie$-isoclinic homomorphism from $(g_1)$ to $(g_2)$ if and only if $\Ker(\beta) \cap \gamma_{n+1}^{\Lie}(\Lieg_1)=0$ and $\Image(\beta) +  {{\ze}_{n}^{\Lie}({\Lieg_2})}  = \Lieg_2.$
\end{Pro}
\begin{proof}
Let $(\alpha,\beta,\gamma) : (g_1) \to (g_2)$ be a $n$-$\Lie$-isoclinic homomorphism induced by $\beta.$ Then  by Proposition \ref{equivalence},  $\Ker(\beta) \cap  \gamma_{n+1}^{\Lie}(\Lieg_1)=0.$ To show that $\Image(\beta)+ {{\ze}_{n}^{\Lie}({\Lieg_2})}  = \Lieg_2,$   let $h\in\Lieg_2.$ Since $\gamma$ is  onto and $(\alpha,\beta,\gamma)$ is a homomorphism of ${n}$-{\Lie}-central extensions,  we have $\tilde{h}=\gamma(\bar{g})=\widetilde{\beta(g)}$ for some $g\in\Lieg_1.$ It follows that $h-\beta(g)\in {\ze}_{\Lie}(\Lieg_2)$ i.e, $h=\beta(g)+z$  for some $z\in {\ze}_{\Lie}(\Lieg_2).$ This proves $\Lieg_2\subseteq\Image(\beta) +  {{\ze}_{n}^{\Lie}({\Lieg_2})},$
the other inclusion being obvious.

 Conversely, suppose that $\Ker(\beta) \cap \gamma_{n+1}^{\Lie}(\Lieg_1)=0$ and $\Image(\beta) +  {{\ze}_{n}^{\Lie}({\Lieg_2})}  = \Lieg_2.$    We claim that the maps $\alpha:=\beta|_{{\ze}_{n}^{\Lie}({\Lieg_1})}$  and $\gamma:\Lieg_1/{{\ze}_{n}^{\Lie}({\Lieg_1})}\to \Lieg_2/{{\ze}_{n}^{\Lie}({\Lieg_2})}$ defined by $\gamma(\bar{{g}})=\widetilde{{\beta(g)}}$ are well-defined homomorphisms for every integer $n.$ We prove by induction that $\beta( {{\ze}_{n}^{\Lie}({\Lieg_1})}  )\subseteq  {{\ze}_{n}^{\Lie}({\Lieg_2})} .$ Let $x\in {{\ze}_{n}^{\Lie}({\Lieg_1})}$ and $h\in\Lieg_2,$ then $h=\beta(g)+z$ for some $g\in\Lieg_1$ and $z\in {{\ze}_{n}^{\Lie}({\Lieg_2})}. $ By inductive hypothesis, we have $\beta([x,g]_{lie}) \in  {{\ze}_{n-1}^{\Lie}({\Lieg_2})}$
as $[x,g]_{lie}\in {{\ze}_{n-1}^{\Lie}({\Lieg_1})},$ so
$$
\begin{aligned}
~~[\beta(x),h]_{lie}&= [\beta(x),\beta(g)+z]_{lie}\\&=[\beta(x),\beta(g)]_{lie}+[\beta(x),z]_{lie}\\&=\beta([x,g]_{lie})+[\beta(x),z]_{lie} \in  {{\ze}_{n-1}^{\Lie}({\Lieg_2})},
\end{aligned}
$$
and thus $\beta(x)\in  {{\ze}_{n}^{\Lie}({\Lieg_2})}.$ One easily verifies that $(\alpha,\beta,\gamma): (g_1)\to (g_2)$  is a homomorphism of ${n}$-{\Lie}-central extensions. It remains to show that it is $n$-\Lie-isoclinic. By Proposition \ref{equivalence}, it is enough to show that $\gamma$  is an isomorphism.  $\gamma$ is onto because  every $h\in\Lieg_2$ can be written as $h=\beta(g)+z$ for some $g\in\Lieg_1$ and $z\in {{\ze}_{n}^{\Lie}({\Lieg_2})},$ yielding $\bar{h}=\widetilde{\beta(g)}=\gamma(\bar{g}).$
To  show that $\gamma$ is one-to-one, let $x\in\Lieg_1$ satisfying $\gamma(\bar{x})=0$, i.e. $\beta(x)\in {{\ze}_{n}^{\Lie}({\Lieg_2})}.$  Let $g_1,\ldots,g_n\in\Lieg_1.$  We need to show that $t:=[[[x,g_1]_{lie},g_2]_{lie},\ldots,g_n]_{lie}=0.$ Indeed, $\beta(t)=[[[\beta(x),\beta(g_1)]_{lie},\beta(g_2)]_{lie},\ldots,\beta(g_n)]_{lie}=0$ since $\beta(x)\in {{\ze}_{n}^{\Lie}({\Lieg_2})}.$ So $t\in\Ker(\beta) \cap \gamma_{n+1}^{\Lie}(\Lieg_1)=0.$ Hence $x\in{{\ze}_{n}^{\Lie}({\Lieg_1})}.$ This completes the proof.
\end{proof}


 \section{Some properties on  $n$-$\Lie$-stem Leibniz algebras} \label{stem section}

\begin{De}
A Leibniz algebra $\Lieg$ is said to be $n$-$\Lie$-stem Leibniz algebra whenever ${{\ze}_{n}^{\Lie}({\Lieg)}}\subseteq \gamma_{n+1}^{\Lie}(\Lieg)$.
\end{De}

The lemma below characterizes $n$-$\Lie$-stem Leibniz algebras.

\begin{Le}\label{stem}
A Leibniz algebra  $\Lieg$ is a $n$-$\Lie$-stem Leibniz algebra if and only if $\LieI\cap\gamma_{n+1}^{\Lie}(\Lieg)\neq0$ for all nonzero two-sided ideal $\LieI$ of $\Lieg.$
\end{Le}
\begin{proof}
Assume that $\Lieg$ is not a $n$-$\Lie$-stem Leibniz algebra and let $g\in {{\ze}_{n}^{\Lie}({\Lieg)}} $ with $g\notin \gamma_{n+1}^{\Lie}(\Lieg).$ Consider the  two-sided  ideal {$\LieI:= \text{span} \{g \}.$} Then $\LieI$ is a nonzero two-sided ideal of $\Lieg$ satisfying $\LieI\cap\gamma_{n+1}^{\Lie}(\Lieg)=0.$

Conversely, suppose that  $\Lieg$ is a $n$-$\Lie$-stem Leibniz algebra and let $\LieI$ be a two-sided  ideal of $\Lieg$ satisfying $\LieI\cap\gamma_{n+1}^{\Lie}(\Lieg)=0.$  Then $\LieI\subseteq {{\ze}_{n}^{\Lie}({\Lieg)}}$ since $[[[x,g_1]_{lie},g_2]_{lie},\ldots,g_{n+1}]_{lie}\in \LieI\cap\gamma_{n+1}^{\Lie}(\Lieg)=0$ for all  $x\in\LieI$ and $g_1,\ldots,g_{n+1}\in\Lieg.$ So $\LieI=\LieI\cap{{\ze}_{n}^{\Lie}({\Lieg)}}\subseteq \LieI\cap\gamma_{n+1}^{\Lie}(\Lieg)=0,$ which implies that $\LieI=0.$
\end{proof}

\begin{Co}\label{existence}
Every Leibniz algebra is $n$-{\Lie}-isoclinic to some $n$-{\Lie}-stem Leibniz algebra.
\end{Co}
\begin{proof}
 Consider the set $\LieM=\{\LieI~|~\mbox{two-sided ideal of}~ \Lieg~\mbox{ satisfying}~ \LieI\cap\gamma_{n+1}^{\Lie}(\Lieg)=0\}.$ This set contains the ideal $\LieI=0.$ So $\LieM$ is  non-empty and  partially ordered by set inclusion.  By Zorn's lemma, it contains a maximal two-sided ideal, call it $\Liem.$ By the property {\it d)} of Lemma \ref{Lem1}, it follows that  $\Lieg$ is $n$-$\Lie$-isoclinic to $\Lieg/\Liem$ since $\Liem\cap\gamma_{n+1}^{\Lie}(\Lieg)=0.$ It remains to show that $\Lieg/\Liem$ is a $n$-$\Lie$-stem Leibniz algebra. Let $\LieI$ be an arbitrary two-sided ideal of $\Lieg$ containing $\Liem$ and  satisfying $\frac{\LieI}{\Liem}\cap\gamma_{n+1}^{\Lie}(\frac{\Lieg}{\Liem})=\Liem.$ By Lemma \ref{stem}, it is enough to show that $\LieI\subseteq\Liem.$  First, we prove that $\LieI\cap\gamma_{n+1}^{\Lie}(\Lieg)=0.$  Let $x\in\LieI\cap\gamma_{n+1}^{\Lie}(\Lieg),$ then $x+\Liem\in\frac{\LieI}{\Liem}\cap\frac{\gamma_{n+1}^{\Lie}(\Lieg)}{\Liem}=\frac{\LieI}{\Liem}\cap\gamma_{n+1}^{\Lie}(\frac{\Lieg}{\Liem})=\Liem.$ So $x\in\Liem$ and thus $x\in\Liem\cap\gamma_{n+1}^{\Lie}(\Lieg)=0.$ Hence $\LieI\cap\gamma_{n+1}^{\Lie}(\Lieg)=0.$
This implies that $\LieI\in\LieM,$ therefore $\LieI\subseteq\Liem$ by maximality of $\Liem.$
\end{proof}
\bigskip

Recall from \cite[Definition 5.4]{Bar}  that the Frattini subalgebra of an algebra {\Lieg} is $\Phi(\Lieg)=\underset{\Liem\in\LieS}\bigcap\Liem$, where $\LieS$ is the set of all maximal subalgebras of $\Lieg.$

 The following classifies  $n$-\Lie-stem Leibniz algebras with trivial Frattini subalgebras.

\begin{Pro}
Let $\Lieg$ be a  Leibniz algebra. If the Frattini subalgebra of $\Lieg$ is trivial, then up to isomorphism, there is only one $n$-\Lie-stem Leibniz algebra that is $n$-\Lie-isoclinic to $\Lieg.$  This $n$-\Lie-stem Leibniz algebra is precisely $\frac{\Lieg}{{{\ze}_{n}^{\Lie}({\Lieg)}}}.$
\end{Pro}
\begin{proof}
Let $\Lies$ be a $n$-\Lie-stem Leibniz algebra that is $n$-\Lie-isoclinic to $\Lieg.$ Then by definition, $\frac{\Lies}{{{\ze}_{n}^{\Lie}({\Lies)}}}\overset{\xi}\cong\frac{\Lieg}{{{\ze}_{n}^{\Lie}({\Lieg)}}}.$ So it is enough to prove that ${{\ze}_{n}^{\Lie}({\Lies)}}=0.$ Let $\Liem$ be a maximal subalgebra of $\Lieg.$ Then either  ${{\ze}_{n}^{\Lie}({\Lieg)}}\subseteq \Liem$ or ${{\ze}_{n}^{\Lie}({\Lieg)}}+\Liem=\Lieg.$ Since $\gamma_{n+1}^{\Lie}({{\ze}_{n}^{\Lie}({\Lieg)}})=0,$ it follows that $ {{\ze}_{n}^{\Lie}({\Lieg})} \cap \gamma_{n+1}^{\Lie}(\Lieg)\subseteq \Liem  + \gamma_{n+1}^{\Lie}(\Liem)=\Liem. $  So by the property {\it b)} of Proposition \ref{xieta}, we have $\xi \left( {{\ze}_{n}^{\Lie}({\Lies})} \cap \gamma_{n+1}^{\Lie}(\Lies)  \right) = {{\ze}_{n}^{\Lie}({\Lieg})}  \cap \gamma_{n+1}^{\Lie}(\Lieg) \subseteq \underset{\Liem\in\LieS}\bigcap\Liem=\Phi(\Lieg)=0.$ Thus ${{\ze}_{n}^{\Lie}({\Lies})} \cap \gamma_{n+1}^{\Lie}(\Lies) =0$ as $\xi$ is an isomorphism. Therefore ${{\ze}_{n}^{\Lie}({\Lies})} =0$ by  Lemma \ref{stem} since $\Lies$ is a $n$-\Lie-stem Leibniz algebra .
\end{proof}

\begin{Th}\label{minimal}
Let  {\Lieg} be a Leibniz algebra. A finite-dimensional Leibniz algebra {\Lieq}, such that ${\Lieq} \underset{n} \sim {\Lieg}$, is a $n$-{\Lie}-stem Leibniz algebra if and only if
${\rm dim}(\mathfrak{q})= {\rm min} \lbrace {\rm dim}(\Lieh)
 | {\Lieh} \underset{n} \sim {\Lieg} \rbrace$.
\end{Th}
\begin{proof}
Let ${\Lieq} \underset{n} \sim {\Lieg}$  and ${\Lieh} \underset{n} \sim {\Lieg}$ be  and assume that {\Lieq}  is a finite-dimensional $n$-{\Lie}-stem Leibniz algebra. Then we have
\begin{align*}
\frac{\gamma_{n+1}^{\Lie}({\Lieh})}{\gamma_{n+1}^{\Lie}({\Lieh}) \cap {\ze}_n^{\Lie}(\Lieh)}&\cong \frac{\gamma_{n+1}^{\Lie}({\Lieh})+{\ze}_n^{\Lie}(\Lieh)}{{\ze}_n^{\Lie}(\Lieh)}\\
&\cong \frac{\gamma_{n+1}^{\Lie}({\Lieh})}{{\ze}_n^{\Lie}(\Lieh)}\\
&\cong \gamma_{n+1}^{\Lie} \left( \frac{{\Lieh}}{{\ze}_n^{\Lie}(\Lieh)} \right)\\
&\cong \gamma_{n+1}^{\Lie} \left( \frac{{\Lieq}}{{\ze}_n^{\Lie}(\Lieq)} \right)\\
&\cong \frac{\gamma_{n+1}^{\Lie}({\Lieq})}{{\ze}_n^{\Lie}(\Lieq)},
\end{align*}
and $\gamma_{n+1}^{\Lie}({\Lieq}) \cong \gamma_{n+1}^{\Lie}({\Lieh})$. So ${\rm dim}\left({\ze}_n^{\Lie}(\Lieq)\right)= {\rm dim} \left(\gamma_{n+1}^{\Lie}({\Lieh}) \cap {\ze}_n^{\Lie}(\Lieh) \right)\leq {\rm dim} \left( {\ze}_n^{\Lie}(\Lieh)\right)$.
On the other hand $\frac{\Lieh}{{\ze}_n^{\Lie}(\Lieh)} \cong \frac{\Lieq}{{\ze}_n^{\Lie}(\Lieq)}$. Therefore ${\rm dim} \left( \Lieq \right) \leq  {\rm dim} \left( \Lieh \right)$.

Conversely, let ${\Lieq} \underset{n} \sim {\Lieg}$  such that {\Lieq} has the minimum dimension. Owing to Corollary \ref{existence} there is a two-sided ideal {\Lieh} of {\Lieq} contained in ${\ze}_n^{\Lie}(\Lieq)$ such that $\Lieq \underset{n}\sim \frac{\Lieq}{\Lieh}$ and ${\ze}_n^{\Lie}(\Lieq)= \left(\gamma_{n+1}^{\Lie}(\Lieq) \cap {\ze}_n^{\Lie}(\Lieq) \right)\oplus  {\Lieh}$. But {\Lieq} has minimum dimension, which implies that $\Lieh=0$, therefore ${\ze}_n^{\Lie}(\Lieq)\subseteq \gamma_{n+1}^{\Lie}(\Lieq)$ and this completes the proof.
\end{proof}

\begin{Th}
If {\Lieg} and {\Lieq} are two $n$-{\Lie}-isoclinic $n$-{\Lie}-stem  Leibniz algebras then ${\ze}_n^{\Lie}(\Lieg)\cong {\ze}_n^{\Lie}(\Lieq)$.
\end{Th}
\begin{proof}
Let {\Lieg} and {\Lieq} be two $n$-{\Lie}-isoclinic $n$-{\Lie}-stem  Leibniz algebras. In view of proof of Theorem \ref{minimal}  and isomorphism $\xi: \gamma_{n+1}^{\Lie} (\Lieg) \longrightarrow \gamma_{n+1}^{\Lie} (\Lieq)$ we have the following commutative diagram with exact rows:
\[
\xymatrix{
0 \ar[r] & {\ze}_n^{\Lie}(\Lieg) \ar[r] \ar[d]^{\xi_{\mid}} & \gamma_{n+1}^{\Lie}(\Lieg) \ar[r] \ar[d]^{\cong} & \frac{ \gamma_{n+1}^{\Lie}(\Lieg) }{{\ze}_n^{\Lie}(\Lieg)} \ar[r] \ar[d]^{\cong} & 0\\
0 \ar[r] & {\ze}_n^{\Lie}(\Lieq) \ar[r]  & \gamma_{n+1}^{\Lie}(\Lieq) \ar[r] & \frac{ \gamma_{n+1}^{\Lie}(\Lieq) }{{\ze}_n^{\Lie}(\Lieq)} \ar[r] & 0
}
\]
where $\xi \left( {\ze}_n^{\Lie}(\Lieg) \right) \subseteq {\ze}_n^{\Lie}(\Lieq)$ since for all $x\in  \gamma_{n+1}^{\Lie}(\Lieg)$, $\eta \left(x+{\ze}_n^{\Lie}(\Lieg)\right) = \xi(x) +  {\ze}_n^{\Lie}(\Lieq)$. Hence, for $x \in {\ze}_n^{\Lie}(\Lieg)$, $0 = \eta  \left( \pi_{\Lieg}(x) \right) = \xi(x) + {\ze}_n^{\Lie}(\Lieq)$ (here $\pi_{\Lieg} : {\Lieg} \twoheadrightarrow {\Lieg}/{\ze}_n^{\Lie}({\Lieg})$ denotes the canonical projection), then $\xi(x) \in {\ze}_n^{\Lie}(\Lieq)$.
Now the Snake Lemma yields  $\xi_{|}$ is a surjective homomorphism and so $\xi ( {\ze}_n^{\Lie}({\Lieg}))= {\ze}_n^{\Lie}({\Lieq})$. Moreover the left hand square is a pull-back diagram, then  $\xi_{\mid}$ is an injective homomorphism. Hence the isomorphism  ${\ze}_n^{\Lie}(\Lieg)\cong {\ze}_n^{\Lie}(\Lieq)$.
\end{proof}


\section*{Acknowledgements}

Second author was supported by  Agencia Estatal de Investigación (Spain), grant MTM2016-79661-P (AEI/FEDER, UE, support included).


\begin{center}

\end{center}

\end{document}